\documentclass[5p, twocolumn]{elsarticle}
\usepackage{graphics} 
\usepackage{epsfig} 
\usepackage{mathptmx} 
\usepackage{times} 
\usepackage{amsmath} 
\usepackage{amssymb}  
\usepackage[utf8]{inputenc}
\usepackage[T1]{fontenc}
\usepackage[section]{placeins}
\usepackage{dsfont}
\usepackage[percent]{overpic}
\usepackage{url}
\usepackage{xcolor,colortbl}

\definecolor{Gray}{gray}{0.85}
\definecolor{LightCyan}{rgb}{0.88,1,1}
\definecolor{lightblue}{rgb}{0.68, 0.85, 0.9}
\definecolor{lightgreen}{rgb}{0.56, 0.93, 0.56}
\definecolor{lightgray}{rgb}{0.83, 0.83, 0.83}
\definecolor{moccasin}{rgb}{0.98, 0.92, 0.84}
\definecolor{gainsboro}{rgb}{0.86, 0.86, 0.86}
\definecolor{pastelgray}{rgb}{0.81, 0.81, 0.77}
\definecolor{lightmauve}{rgb}{0.86, 0.82, 1.0}
\definecolor{lavendergray}{rgb}{0.77, 0.76, 0.82}
\definecolor{lavenderblue}{rgb}{0.8, 0.8, 1.0}
\newcolumntype{a}{>{\columncolor{Gray}}c}
\newcolumntype{b}{>{\columncolor{white}}c}
\definecolor{navajowhite}{rgb}{1.0, 0.87, 0.68}
\definecolor{languidlavender}{rgb}{0.84, 0.79, 0.87}
\definecolor{gainsboro}{rgb}{0.86, 0.86, 0.86}
\definecolor{celadon}{rgb}{0.67, 0.88, 0.69}
\definecolor{beaublue}{rgb}{0.74, 0.83, 0.9}
\definecolor{paleaqua}{rgb}{0.74, 0.83, 0.9}
\definecolor{pearl}{rgb}{0.94, 0.92, 0.84}
\definecolor{periwinkle}{rgb}{0.8, 0.8, 1.0}
\definecolor{mossgreen}{rgb}{0.68, 0.87, 0.68}
\definecolor{forestgreen}{rgb}{0.13, 0.55, 0.13}
\usepackage{algorithmic}
\usepackage{color}
\usepackage[]{algorithm2e}
\usepackage{colortbl}
\newtheorem{definition}{Definition}
\newtheorem{assumption}{Assumption}
\newtheorem{proposition}{Proposition}

\newtheorem{thm}{Theorem}
\newtheorem{cor}{Corollary}
\newtheorem{proof}{Proof}


\usepackage[export]{adjustbox}

\definecolor{almond}{rgb}{0.94, 0.87, 0.8}
\definecolor{aliceblue}{rgb}{0.94, 0.97, 1.0}
\definecolor{grannysmithapple}{rgb}{0.66, 0.89, 0.63}
\definecolor{white}{rgb}{0.9, 0.9, 0.98}

\usepackage{lineno,hyperref}
\modulolinenumbers[10]

\journal{European Journal of Control}

\begin{document}

\begin{frontmatter}
\title{ On the probabilistic feasibility of solutions in multi-agent optimization problems under uncertainty}
\author[1]{George Pantazis}\fntext[myfootnote]{Corresponding author}
\ead{georgios.pantazis@lmh.ox.ac.uk}
\author{Filiberto Fele}
\ead{ filiberto.fele@eng.ox.ac.uk} 
\author{Kostas Margellos}
\ead{kostas.margellos@eng.ox.ac.uk}
\address{Department of Engineering Science, University of Oxford, OX1 3PJ, UK}
%


\begin{abstract}
We investigate the probabilistic feasibility of randomized solutions to two distinct classes of uncertain multi-agent optimization programs.  We first assume that only the constraints of the program are affected by uncertainty, while the cost function is arbitrary. Leveraging recent \emph{a posteriori} developments of the scenario approach, we  provide probabilistic guarantees for all feasible solutions of the program under study. This result is particularly useful in cases where numerical difficulties related to the convergence  of the solution-seeking algorithm hinder the exact quantification of the optimal solution. Furthermore, it can be applied to cases where the agents' incentives lead to a suboptimal solution, e.g., under a non-cooperative setting. We then focus on optimization programs where the cost function admits an aggregate representation and depends on uncertainty while constraints are deterministic. By exploiting
the structure of the program under study and leveraging the so called support rank notion, we
provide agent-independent robustness certificates for the optimal solution, i.e., the
constructed bound on the probability of constraint violation
does not depend on the number of agents, but only on the
dimension of the agents' decision.
This substantially reduces the number of samples
required to achieve a certain level of probabilistic robustness
as the number of agents increases. 
All robustness certificates provided in this paper are distribution-free and can be used alongside any optimization algorithm. Our theoretical results are accompanied by a numerical case study involving a charging control problem of a fleet of electric vehicles. 
\end{abstract}
\end{frontmatter}
\section{Introduction}
\subsection{Background}
A vast amount of today's challenges in the domains of energy systems \cite{Ioli2017}, \cite{Gharesifard2016}, traffic networks \cite{Dario2019}, economics \cite{Acemoglu2013} and the social sciences \cite{Ghaderi}, \cite{Acemoglu} revolve around multi-agent systems, i.e., systems which comprise different entities/agents that interact with each other and make decisions, based on individual or collective criteria. Existing literature provides a plethora of methods to solve such problems. Each method is appropriately designed to fit the structure of these interactions and the agents' incentives. To address challenges, related to the computational complexity issues and privacy concerns of solving a multi-agent optimization problem in a centralised fashion, several decentralised or distributed coordination schemes have been proposed \cite{BertsekasParallel},\cite{Parikh}. In the decentralised case, agents optimize their cost function locally and then communicate their strategies to a central authority. In the distributed case, a central authority is absent and agents communicate with each other over a network, exchanging information with agents considered as neighbours given the underlying communication protocol.  In either case, the presence of uncertainty in such problems constitutes a critical factor that, if not taken into account, could lead to unpredictable behaviour, hence it is of major importance to accompany the solutions of such algorithms with robustness certificates. In this paper we assume that the probability distribution of the uncertainty is considered unknown and adopt a data-driven approach, where the uncertainty is represented by means of scenarios that could be either available as historical data, or extracted via some prediction model. To this end, we work under the framework of the so called scenario approach. \par The scenario approach is  a well-established mathematical technique \cite{CampiCalafiore2006}, \cite{CampiGaratti2008}, \cite{Campi2008},  and  still a highly active research area (see \cite{Campi2018}, \cite{CampiGarattiRamponi2018} for some recent developments), originally introduced to provide \emph{a priori} probabilistic guarantees for solutions of uncertain convex optimization programs. Recently, the theory has been extended to non-convex decision making problems \cite{Campi2018}, \cite{CampiGarattiRamponi2018} where the probabilistic guarantees are obtained in an \emph{a posteriori} fashion. The main advantage of the scenario approach is its applicability under very general conditions, since it does not require the knowledge of the uncertainty set or the probability distribution, unlike robust \cite{Bai} and stochastic optimization \cite{BirgLouv97}. According to the scenario approach, the original problem can be approximated by solving a computationally tractable approximate problem, the so called scenario program consisting of a finite number of constraints, each of them corresponding to a different realization of the uncertain parameter. In this realm we present some of  the challenges that pertain uncertain multi-agent systems.

\begin{table*}[!]
	\setlength{\arrayrulewidth}{0.5pt}
	\centering
	\resizebox{\linewidth}{!}{
	\begin{tabular}{|p{3.5cm}|p{2.2cm}|p{3.4cm}|p{5cm}|p{3cm}|}
		
		\hline
		\ & \  & \ &  \  & \ \\
		\normalsize{\  Type of solution} & \normalsize Nature of certificate & \normalsize \ \  Scalability & \normalsize \ \ \ \  \ \  \ \  \ Class of problems & \normalsize  \ \ \  \ \ \  \ \ Result \\ 
		\hline 
		
		\textbf{Entire feasible set} & \textbf{\emph{A posteriori}}  & \textbf{Agent dependent} & \textbf{Feasibility programs with uncertain convex constraints} & \textbf{Thm. 2} \\
		
		\hline 
		Subset of feasible solutions & \emph{A priori}  & Agent dependent & Feasibility programs with uncertain convex constraints & Thm. 2 of \cite{GrammaticoNonConvex}\\
		\hline
		Variational inequality solution set  &\emph{A posteriori} & Agent dependent  & Variational inequality problems with uncertain convex constraints & Thm. 1 of \cite{Fabiani2020b} \\ 
		\hline
		Variational inequality solution set  &\emph{A posteriori} & Agent dependent  & Aggregative games with uncertain affine constraints & Thm. 1 of \cite{Fabiani2020a} \\
		\hline 
		\hline
		\textbf{Unique optimizer} & \textbf{\emph{A priori}} & \textbf{Agent independent}  &  \textbf{Optimization programs with uncertain aggregative term in the cost} & \textbf{Thm. 3} \\ 
		\hline 
		Unique variational inequality solution & \emph{A posteriori} \& \emph{a priori} &  Agent dependent  & Variational inequalities with uncertain cost or constraints &  Cor. 1 of \cite{DarioVI}, Thm. 9 of \cite{Fele2019}, Thm. 5 of \cite{Feleconf2019} (only \emph{a posteriori}) \\
		\hline
	\end{tabular}}
\caption{\normalsize Classification of main results according to their main features and comparison with existing literature}

\end{table*}

\subsection{Challenges to be addressed and main contributions}
In many problems of practical interest, agents' decisions are feasible thus satisfying the imposed constraints, however they are not necessarily optimal. This can be due to a variety of reasons, such as:
\begin{enumerate}
	\item  Numerical difficulties related to the convergence properties of the solution-seeking algorithm that hinder the exact quantification of the optimal solution;
	\item  The nature of agents' incentives. For example, when a  non-cooperative scheme is adopted, the social welfare optimum cannot be usually achieved, thus resorting to suboptimal solution concepts, such as equilibria; 
\end{enumerate} \par
However, even for cases where the nature of the problem allows for the quantification of a social welfare optimum with high precision, a major challenge in investigating the probabilistic robustness of such solutions is the dependence of the provided certificates on the number of agents. Given that we wish to obtain identical probabilistic guarantees as the population increases,  a larger number of samples is required. This fact renders the provided guarantees conservative for large scale applications. As such, it is of utmost importance to show that for certain classes of problems, common in practical applications, the obtained probabilistic guarantees can be agent independent.
The main contributions of this paper are as follows: \\ 
\begin{enumerate}
	\item We address the challenges related to the lack of optimality of the agents' decisions, by considering a set-up, where the choice of the cost function can be arbitrary and the uncertainty affects only the constraints. We, then, leverage recent results of \cite{CampiGarattiRamponi2018} in order to provide \emph{a posteriori} robustness certificates for the entire feasibility region\footnote{In this paper we will focus on optimization programs affected by uncertainty. It is important to mention, though, that our results are applicable to more general uncertain feasibility programs. This allows their use for providing certificates for the feasibility region of other classes of problems, as well, such as variational inequality problems and generalised Nash equilibrium problems \cite{fele-a}.}. The theoretical framework of \cite{fele-a}, initially developed only for uncertain feasibility problems with polytopic constraints, is extended in our set-up to account for general (possibly coupling) uncertain convex constraints.	An immediate consequence of this result is that distribution-free probabilistic guarantees for the entire set of solutions to optimization programs can be provided, thus circumventing the need for selecting a unique optimizer via the enforcement of a tie-break rule.   
	\item We then focus on a specific class of uncertain  multi-agent programs, prevalent in many practical applications, where the cost is considered to be a function of the aggregate decision and affected by uncertainty. A similar problem formulation has been considered  in \cite{Deori2016} and is extended to our set-up to account for the presence of uncertainty in the cost function. Other problems whose structure shares similarities with our work can be found in \cite{gram1},\cite{gram2}, \cite{gram3}, though under a purely deterministic and game-theoretic set-up. 
	Following the recent developments in \cite{feleb} we show based on the notion of the support rank \cite{MorariSupportRank}, that the obtained probabilistic feasibility certificates are agent independent. This result directly outperforms probabilistic feasibility statements obtained by a direct application of the scenario approach theory \cite{CampiGaratti2008}, and has superior scalability properties as it does not depend on the number of agents involved.

\end{enumerate}

The contributions of our main results in comparison with results in the literature are summarised in Table 1. Our first contribution provides probabilistic guarantees, in an \emph{a posteriori} fashion, that hold for the entire feasibility region in contrast with the \emph{a priori} result in \cite{GrammaticoNonConvex}, where the construction of a feasible subset from the convex hull of randomized optimizers may produce a "thin" subset of the feasible region. Another recent contribution \cite{Fabiani2020b}  uses a similar approach to \cite{fele-a} to provide guarantees applicable not for the entire feasibility region but for the solution set of uncertain variational inequalities in an \emph{a posteriori} fashion. For our second contribution, the nature of the certificates is \emph{a priori}. Assuming uniqueness of the solution, we exploit the aggregative structure of the cost  to provide results that are agent independent. The aggregative nature of the cost has been exploited in several works (see \cite{Fele2019}, \cite{Feleconf2019} and \cite{Fabiani2020a}), however, guarantees provided in these works are dependent on the number of agents. It should be apparent from Table 1 that our results are the first of their kind to provide guarantees for the entire feasibility region, as well as the first agent independent result for a particular class of optimization programs. \par

The rest of the paper is organized as follows:
In  Section II  probabilistic guarantees for all feasible solutions of optimization programs with arbitrary cost and uncertain convex constraints are provided. Section III focuses on providing agent-independent robustness certificates for the optimal solution set of a specific class of aggregative optimization programs with uncertain cost. The aforementioned results are used in Section IV in the context of a numerical study on the charging control problem for a fleet of electric vehicles. Section V concludes the paper and provides some potential future research directions.

\subsection{Notation}
Let $\mathcal{N}=\{1, \dots, N\}$ be the index set of all agents, where $N$ denotes their total number and $x_i$ the strategy of agent $i$ taking values in the deterministic set $X_i \subseteq \mathbb{R}^{n}$. 
We denote $x=(x_i)_{i \in \mathcal{N}} \in X =  \prod_{i \in \mathcal{N}}X_i \subseteq \mathbb{R}^{nN}$ the collection of all agents' strategies and $bdry(X)$ the boundary of a set $X$. Similarly, the vector $x_{-i}=(x_j)_{j \in \mathcal{N}, j \neq i} \in \prod_{j \in \mathcal{N}, j \neq i}X_j \subseteq \mathbb{R}^{n(N-1)}$ denotes the collection of the decision vectors of all other agents' strategies except for that of agent $i$. Let $\theta$ be an uncertain parameter defined on the (possibly unknown) probability space $(\Theta,\mathcal{F}, \mathbb{P})$, where $\Theta$ is the sample space, equipped with a $\sigma$-algebra   $\mathcal{F}$  and a probability measure  $\mathbb{P}$.  
Furthermore, let $\{\theta_{m}\}_{{m \in \mathcal{M}} } \in \Theta^M$, $\mathcal{M}=\{1,\dots, M\}$ be a finite collection of $M$ independent and identically distributed (i.i.d.) scenarios/realisations of the uncertain vector $\theta$, where $\Theta^M$ is the cartesian product of multiple copies of the sample space $\Theta$. Finally, $\mathbb{P}^{M}= \prod_{m \in \mathcal{M}} \mathbb{P}$  is the associated product probability measure. The symbols $x$ and $(x_i, x_{-i})$ are used interchangeably in this paper, depending on the context.

\section{Optimization programs with uncertainty in the constraints} 
\label{collective_feasibility}
\subsection{The convex case} 
Consider the following optimization program
\begin{align}
\mathrm{P}_{\Theta}: \min_{x \in X} J(x) \ \text{subject to } \ x \in  \bigcap_{\theta \in \Theta}X_{\theta}, \label{feasibility_problem}
\end{align}
where the cost function $J(x)$ can be chosen arbitrarily (even feasibility programs would be admissible), $x$ is a decision vector taking values in $X \subset\mathbb{R}^{d}$ and  $X_\theta$ is dictated by the uncertain parameter $\theta$.
We seek to provide probabilistic guarantees for all feasible solutions of this class of programs. 

For the optimization program (\ref{feasibility_problem}), we define the following  scenario program, where the multi-sample $\{\theta_{m}\}_{m \in \mathcal{M}}$ is assumed to be drawn in an i.i.d. fashion from $\Theta^M$. 
\begin{align} \label{scenario_feasibility}
\mathrm{P}_{M}: \min_{x \in X} J(x) \ \text{subject to } x  \in \bigcap_{m \in \mathcal{M}} X_{\theta_{m}}.
\end{align}
Our results depend on a convex constraint structure, thus we impose the following assumption: 
\begin{assumption} \label{ass1} 
	\begin{enumerate}
		\item The deterministic constraint set $X$  is a non-empty, compact and convex set.
		\item For any random sample $\theta \in \Theta$,  we have that
		$X_{\theta}=\{x \in \mathbb{R}^{d} : u(x, \theta) \leq 0 \}$, where $u: \mathbb{R}^{d} \times \Theta\rightarrow \mathbb{R}^{q}$ is a vector-valued convex function. 
		\item For any fixed  multi-sample $\{\theta_{m}\}_{m \in \mathcal{M}}$ the convex set $C_M= \{\bigcap_{{m \in \mathcal{M}}} X_{\theta_{m}}\} \bigcap X=\{x \in X : u(x, \theta_{m}) \leq 0, \forall  \ m \in \mathcal{M}\} $  has a non-empty interior.
	\end{enumerate}
\end{assumption}
Assumption \ref{ass1} guarantees that $P_M$ admits at least one solution for any chosen multisample  $\{\theta_m\}_{m \in \mathcal{M}}$.

The optimization program $P_M$ can be equivalently written as $\min J(x) \ \text{subject to} \ x \in C_M.$
Upon finding the feasibility domain $C_M$ of the problem $P_M$, we are interested in investigating the robustness properties collectively for all the points of this domain to yet unseen samples, in other words in quantifying the probability that a new sample $\theta \in \Theta$ is drawn such that the constraint $X_\theta$ defined by this sample is not satisfied by some point $x\in C_M$. This concept, which is of crucial importance for our work, is known in the literature as the probability of violation and  is adapted in our context to represent the probability of violation \emph{of an entire set}. By Definition 1 in \cite{CampiCalafiore2006} 
the probability of violation of a given point $x \in C_M$ is defined as:
\begin{align}
& V(x)=  \mathbb{P} \Big \{ \theta \in \Theta :~ x \notin X_\theta\Big \}. \label{Violationpoint}
\end{align} 
By Assumption \ref{ass1}, the probability of violation can be equivalently written as $V(x)=  \mathbb{P} \Big \{ \theta \in \Theta :~ u(x,\theta)>0 \Big \}$. We can now define the probability of violation of the entire convex set $C_M$.
\begin{definition}  \label{def2}
	Let  $\mathcal{C} \subseteq 2^X$  be the set of all non-empty, compact and convex sets contained in $X$. For any $C_M \in \mathcal{C}$ we define  the probability of violation of the set $C_M$  as  a mapping $\mathbb{V}: \mathcal{C} \rightarrow [0,1]$ given by the following relation: 
	\begin{align*} 
	\mathbb{V}(C_M)&= \sup_{x \in C_M} V(x).
	\end{align*} 
\end{definition} 
In Definition \ref{def3} below two concepts of crucial
importance are introduced.
\begin{definition} \label{def3}
	\begin{enumerate}
		\item  For any $M$, an algorithm is a mapping $A_M: \Theta^M \rightarrow \mathcal{C} \subseteq 2^X$ that associates the multi-sample $\{\theta_m\}_{m \in \mathcal{M}}$ to a unique convex set $C_M \in \mathcal{C}$. 
		\item	Given a multi-sample $\{\theta_{m}\}_{m \in \mathcal{M}} \in \Theta^M$, a set of samples $\{\theta_{m}\}_{m \in I_k} \subseteq \{\theta_{m}\}_{m \in \mathcal{M}}$, where $I_k=\{m_1,m_2,\dots,m_k\}$  is called a support subsample if  $A_k(\{\theta_{m}\}_{m \in I_k})=A_M(\{\theta_{m}\}_{m \in \mathcal{M}})$ i.e., the solution returned by the algorithm when fed with $\{\theta_{m}\}_{m \in I_k}$ coincides with the one obtained when the entire multi-sample is used. The support subsample with the smallest cardinality among all the possible support subsamples, is known as the minimal support subsample. 
		\item  A support subsample function is a mapping of the form $B_M: \{\theta_m\}_{m \in \mathcal{M}}  \rightarrow$ 
		$\{m_1,m_2,\dots,m_k\}$  that takes as input all the samples and returns as output only the indices of the samples that belong to the support subsample. 
	\end{enumerate}
\end{definition}
Note that the notions of support subsample and support subsample function in Definitions \ref{def3}.2, \ref{def3}.3 are respectively referred to as compression set and compression function in \cite{Connection}. 
The cardinality of the support subsample $\{\theta_{m}\}_{m \in I_k}$ is by definition the cardinality of the output of the function $B_M$, namely,  $k=|B_M(\{\theta_m\}_{m \in \mathcal{M}})|$. 
In our set-up the minimal support subsample consists of the indices of the samples that are of support for the entire feasibility region, according to the following definition 
\begin{definition} (Support sample for the feasibility region)
	A sample $\theta_j \in \{\theta_m\}_{m \in \mathcal{M}}$ is said to be of support for the feasible set $C_M$ of $\mathrm{P}_{M}$  if its removal leads to an enlargement of the feasible region, i.e., when $\{ \bigcap_{m \in \mathcal{M}} X_{\theta_m}\}\setminus X_{\theta_j} \supset C_M$ or, equivalently, when the set
	$bdry(X_{\theta_j})\bigcap C_M$ is non-empty. 
\end{definition} \par
The number of support samples of the feasible region or, equivalently, the cardinality of the minimal support subsample is denoted as $F_M$.
The constraints that correspond to indices from the minimal support subsample can be alternatively viewed as an extension of the notion of the facets of a polytope (see Definition 2.1 in \cite{Ziegler1995}) adapted to the more general case of compact and convex sets. Note that a single constraint may give rise to multiple "facets". Figure 1  illustrates the concept of the minimal support subsample by showing the feasible region formed by random convex constraints. Note that only the indices of the samples $\theta_3, \theta_5, \theta_6, \theta_7$ belong to the minimal support subsample, since if we feed only these samples as input into the algorithm $A_M$ the feasible set $C_M$ is returned.
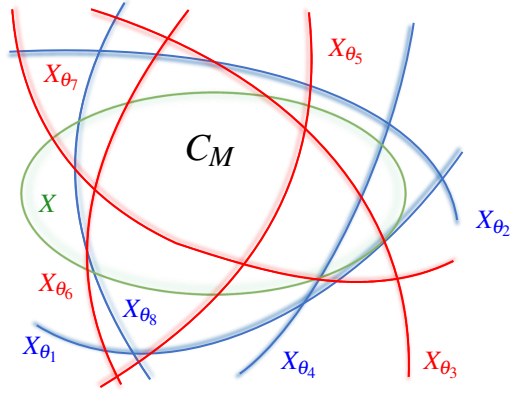
\begin{figure}
	\centering
	\begin{overpic}[clip, trim=10cm 3.5cm 10cm 5cm, scale=0.5]{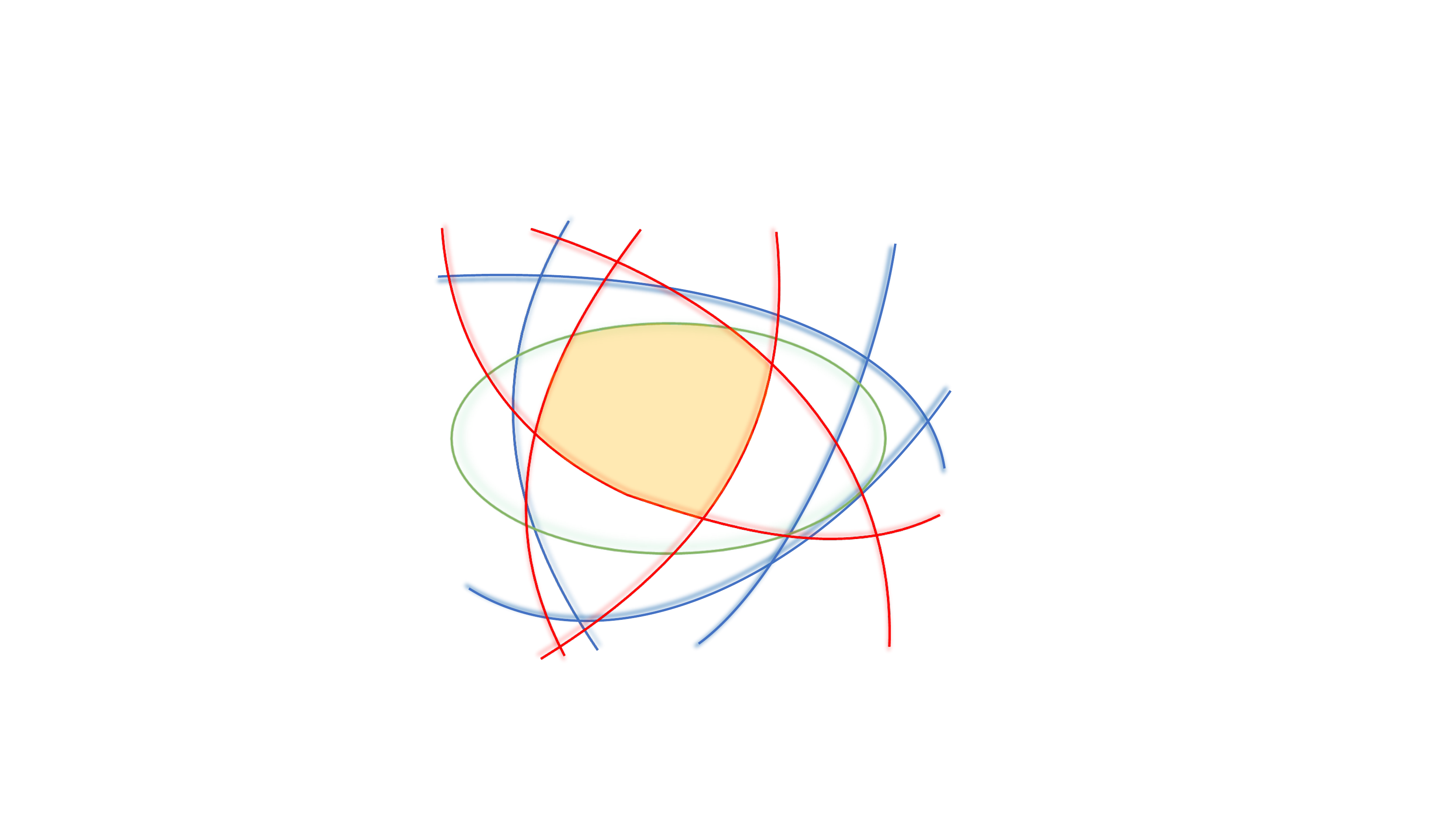}
		\put (4,8) {$\displaystyle \color{blue} X_{\theta_1}$}
		\put(90,32) {$\displaystyle \color{blue} X_{\theta_2}$}
		\put(80,5) {$\displaystyle \color{red} X_{\theta_3}$}
		\put(53,5) {$\displaystyle \color{blue} X_{\theta_4}$}
		\put(7,35) {$\displaystyle \color{forestgreen} X$}
		\put(62,65) {$\displaystyle \color{red} X_{\theta_5}$}
		\put(7,20) {$\displaystyle \color{red} X_{\theta_6}$}
		\put(8,60) {$\displaystyle \color{red} X_{\theta_7}$}
		\put(23,15) {$\displaystyle \color{blue} X_{\theta_8}$}
		\put(35,45) {\Large $\displaystyle \color{black}  C_M $}
	\end{overpic}
	\caption{ The feasibility region $C_M$ and its connection with random convex constraints produced by eight i.i.d. samples $\{\theta_1, \theta_2, \dots, \theta_8\}$. The constraint in green corresponds to the deterministic constraint $X$. Note that only the indices of the samples $\theta_3, \theta_5, \theta_6, \theta_7$ belong to the minimal support subsample since their corresponding constraints (in red) form along with the deterministic constraint $X$ the boundary of $C_M$.}
\end{figure}



Another important notion used in our work is the notion of an extreme point. An extreme point can be viewed as an extension of the vertex of a polytope for arbitrary compact and convex sets and is defined as a point which is not in the interior of any line segment lying entirely in the set. This property is formally presented in the following definition.
\begin{definition} (Extreme points) \cite{simon2011}.
	An extreme point of a convex set C is a point $x \in C$ for which the following property holds:
	If $x$ can be written as a convex combination of the form $x=\lambda x_1+(1-\lambda)x_2$ with $x_1,x_2 \in C$ and $\lambda \in [0, 1]$, then $x_1=x$ and/or $x_2=x$.
\end{definition}
Note that the number of extreme points of a convex set depends on the geometry of the set under study and can also be infinite, e.g., in the case of a $d$-dimensional ball.

Our work focuses on compact and convex sets defined over a finite-dimensional space, where the following theorem can be applied.
\begin{thm} (Minkowski – Caratheodory Theorem) \cite{simon2011}.
	Let $C$ be a compact convex
	subset of $\mathbb{R}^d$ of dimension $d$. Then any point in $C$ is a convex combination of at
	most $d + 1$ extreme points.
\end{thm}

We equip the set of extreme points of the convex set $C_M$ with indices and denote this set of indices $E(C_M)$, while $\text{bdry}(C_M)$ refers to the boundary of $C_M$. It is important to emphasize that the dependence of the convex set $C_M$ on the multi-sample $\{\theta_m\}_{m \in \mathcal{M}}$ implies that $|B_M|$ (and $|E(C_M)|$) are random variables that depend on $\{\theta_m\}_{m \in \mathcal{M}}$. 

Next we define the set 
\begin{align}
\mathcal{C}_\theta&=\{C \in \mathcal{C}:  u(x_j,\theta) \leq 0, \  \forall \ j \in E(C) \} \nonumber \\
&=\{C \in \mathcal{C} : C \subseteq X_\theta \} \label{P},
\end{align}
of all the non-empty, compact and convex sets where elements $C$ satisfy the constraint associated with the sample $\theta \in \Theta$. Note that if all the extreme points of the set satisfy the inequality $u(\cdot,\theta) \leq 0$, then every point $x\in C$ of the set satisfies it as well. To see this, note that $x$ can always be expressed as a convex combination of the set's extreme points. \par 

Our aim is to provide probabilistic guarantees for a non-empty, convex and compact set $C_M$ constructed by the intersection of  $M$ random realizations of the uncertain convex constraint $X_\theta=\{x \in \mathbb{R}^d: u(x,\theta) \leq 0 \}$, where $u: \mathbb{R}^{d} \times \Theta \rightarrow \mathbb{R}^{q}$ is a convex function with respect to the decision variable $x$.  
\begin{thm} \label{Theorem1}
	Consider Assumption \ref{ass1} and any $A_M, B_M$ as in Definition \ref{def3}. Fix $\beta \in (0,1)$ and define the violation level function $\epsilon: \{0,...,M\} \rightarrow [0,1]$ such that
	\begin{align}
	\epsilon(M)=1 \ \text{and} \  \sum_{k=0}^{M-1} {M \choose k}(1-\epsilon(k))^{M-k}=\beta. \label{epsilon}
	\end{align} 
	We have that 
	\begin{align*}
	\mathbb{P}^{M} \Big \{\{\theta_m\}_{m \in \mathcal{M}} \in \Theta^{M}:~\mathbb{V}(C_M)> \epsilon(k^*) \Big \} \leq \beta,
	\end{align*}
	where $k^*=F_M$ is the number of support samples according to Definition 3.
\end{thm}
\begin{proof}
	For a fixed multisample  $ \{\theta_m\}_{m \in \mathcal{M}} \in \Theta^{M}$ consider an arbitrary point $x \in C_M$. Then,  the following inequalities are satisfied
	\begin{align}
	&  V(x)  =  \mathbb{P} \Big \{ \theta \in \Theta :~ x  \notin X_\theta \Big \}=  \mathbb{P} \Big \{ \theta \in \Theta :~u(x,\theta)>0\Big \} \nonumber \\
	&\stackrel{(i)}{=} \mathbb{P} \Big \{ \theta \in \Theta :~ u( \sum_{j \in I_{d+1}} {\lambda_j x_j},\theta)>0\Big \}  \nonumber \\
	& \stackrel{(ii)}{\leq} \mathbb{P} \Big \{ \theta \in \Theta :~ \sum_{j \in I_{d+1}}\lambda_j  u( x_j,\theta)>0\Big \}  \nonumber \\ 
	& \leq \mathbb{P} \Big \{ \theta \in \Theta :~  \sum_{j \in I_{d+1}}\lambda_j \max_{j \in I_{d+1}} u( x_j,\theta)>0\Big \}  \nonumber \\
	& \leq \mathbb{P} \Big \{ \theta \in \Theta :~ \max_{j \in I_{d+1}} u( x_j,\theta) >0\Big \}   \nonumber  \\
	& = \mathbb{P}\bigg \{ \bigcup_{j \in I_{d+1}} \Big \{ \theta \in \Theta :~ u(x_j,\theta) >0\Big \} \bigg \}  \nonumber   \\
	&\stackrel{(iii)}{\leq} \mathbb{P}\bigg \{ \bigcup_{j \in E(C_M)} \Big \{ \theta \in \Theta :~ u(x_j,\theta) >0\Big \} \bigg \},  \label{result1}
	\end{align}
	Equality (i) is derived from Theorem 1, where the set under study is the convex set $C_M$. In our case, the Minkowski-Caratheodory theorem states that any arbitrary point of the set $x \in C_M$  can be represented as a convex combination of at most $d+1$  extreme points of $C_M$, which means that there exists a subset of extreme points  $ \{x_j\}_{j \in I_{d+1}} \subseteq \{x_j\}_{j \in E(C_M)} $  such that  $x=\sum_{j \in I_{d+1}} {\lambda_j x_j}$, where $\sum_{j \in I_{d+1}} \lambda_j=1$ and $\lambda_j \geq 0, \ \forall \ j \in I_{d+1}$. Equality (ii) stems from the fact that $u$ is a convex function of $x$ for any given $\theta \in \Theta$. The last inequality follows from the fact that $I_{d+1}$ is a set of indices corresponding to extreme points and as such is a subset of $E(C_M)$. Since (\ref{result1}) holds for all $x \in C_M $, it can equivalently be written as 
	\begin{equation}
	\mathbb{V}(C_M)=\sup_{x \in C_M} V(x) \leq \mathbb{P}\bigg \{ \bigcup_{j \in E(C_M)} \Big \{ \theta \in \Theta :~ u( x_j,\theta) >0\Big \} \bigg \}. \nonumber
	\end{equation}
	Therefore, for any multi-sample $\{\theta_m\}_{m \in \mathcal{M}}$ and for any cardinality (not necessarily minimal) of the support subsample $k \in \{1,...,N\}$ the following inequalities are satisfied:
	\begin{align} 
	&\mathbb{P}^{M} \Big \{  \{\theta_m\}_{m \in \mathcal{M}} \in \Theta^{M} :~ \mathbb{V}(C_M) > \epsilon(k^*) \Big \} \nonumber \\
	&\leq \mathbb{P}^{M} \Big \{  \{\theta_m\}_{m \in \mathcal{M}} \in \Theta^{M} \!\!:~ 
	\!\!\!\! \nonumber \\
	& \ \ \ \ \ \ \ \ \ \ \ \ \mathbb{P}\Big \{ \bigcup_{j \in E(C_M)}  \Big \{ \theta \in \Theta : u(x_j,\theta)>0 \Big \}> \epsilon(k^*) \Big \} \nonumber\\
	&= \mathbb{P}^{M} \! \Big \{\! \{\theta_m\}_{m \in \mathcal{M}}\! \!\in \Theta^{M}\! :\!\!\!\!~ \nonumber \\
	& \ \ \ \  \ \  \ \  \ \ \  \ \mathbb{P} \Big \{ \!\!\theta \in \Theta\! : \ \exists \ \!\!  j \in E(C_M), u(x_j,\theta)\!\!>\!\! 0 \Big \}\!\!\!>\!\epsilon(k^*) \Big \} \nonumber\\
	&=\mathbb{P}^{M} \Big \{  \{\theta_m\}_{m \in \mathcal{M}} \in \Theta^{M}:~ \nonumber \\
	& \ \ \ \  \ \ \  \ \ \  \ \ \mathbb{P} \Big \{ \theta \in \Theta :~ C_M \not\subseteq X_\theta\Big \}  > \epsilon(k^*) \Big \}, \label{first}
	\end{align}
	where the last equality is due to (\ref{P}). Define now an algorithm $A_M$ as in Definition  \ref{def3}.1, that returns the convex set confined by the feasibility region of $C_M$. By construction, $A_M$ satisfies Assumption 1 of \cite{CampiGarattiRamponi2018},  since for any multi-sample $\{\theta_m\}_{m \in \mathcal{M}}$ it holds that $A_M(\{\theta_m\}_{m \in \mathcal{M}}) \in \mathcal{C}_{\theta_m}$, for all $m \in \mathcal{M}$. The satisfaction of Assumption 1  paves the way for the use of Theorem 1 of \cite{CampiGarattiRamponi2018}. 
	In particular, Theorem 1 of \cite{CampiGarattiRamponi2018} implies that the right-hand side of (\ref{first}) can be upper bounded by $\beta$.
	
	As such, we have that
	\begin{align}
	&\mathbb{P}^{M} \Big \{\{\theta_m\}_{m \in \mathcal{M}} \in \Theta^{M}\!\!:\!\!~\mathbb{P} \Big \{ \theta \in \Theta :~\!\!\!\! C_M \not\subseteq X_\theta\Big \}  > \epsilon(k^*) \Big \} = \nonumber \\
	&\mathbb{P}^{M} \Big \{ \{\theta_m\}_{m \in \mathcal{M}} \in \Theta^{M}\!\!:\!\!~\mathbb{P} \Big \{ \theta \in \Theta :~\!\!\!\! C_M \notin \mathcal{C}_\theta \Big \}  > \epsilon(k^*) \Big \} \leq \beta. \label{second}
	\end{align}
	From (\ref{first}) and (\ref{second}) we obtain that:
	\begin{align}
	\mathbb{P}^{M} \Big \{ \{\theta_m\}_{m \in \mathcal{M}} \in \Theta^{M}:~\mathbb{V}(C_M) > \epsilon(k^*) \Big \}  \leq \beta,
	\end{align}
	thus concluding the proof.
\end{proof}

The result of Theorem \ref{Theorem1} implies that with confidence at least $1-\beta$, the probability that there exists at least one feasible solution of $C_M$ that violates the constraints for a new realization $\theta \in \Theta$, is at most equal to $\epsilon(k^*)$. 	Note that our guarantees trivially hold for any subregion of the feasible set. However, the support subsample cannot be easily computed in the general case. Restricting our attention to programs subject to uncertain affine constraints, provides the means to quantify the support subsample.  

\subsection{The polytopic case}
Assuming the presence of affine constraints only, we replace Assumption \ref{ass1} with the following:
\begin{assumption} \label{polytope}
	Consider Assumption \ref{ass1} and further assume that $X$ is polytopic\footnote{A polytope $\Pi \in  \mathbb{R}^{d}$ can be expressed by its H-representation, i.e., the intersection of a finite number of halfspaces, and also as the convex hull of its vertex set $v(\Pi)=\{x_1,...,x_Q \}  $ i.e, $\Pi=conv(v(\Pi))=\{ \sum_{j=1}^{Q}{x_j\lambda_j} : \sum_{j=1}^{Q}{\lambda_j}=1,  \lambda_j \geq 0, \ j=1,...,Q\}$, where $v(\cdot)$ and $conv(\cdot)$ denote the set of vertices of the polytope  and the convex hull, respectively. This representation is generally known as $V$-representation.} and $u(x,\theta) = a^Tx - b \leq 0 $, where $a \in \mathbb{R}^{d}$, $b \in \mathbb{R}$ and $\theta=(a^T \  b) \in \mathbb{R}^{d+1}$.
\end{assumption}
Since all feasibility sets are now polytopic, rather than $C$ and $C_M$, we use $\Pi$ and $\Pi_M$, respectively. 
Under Assumption \ref{polytope}, the cardinality of the minimal support subsample coincides by definition with the number of random facets (see Definition 2.1 in \cite{Ziegler1995}) of the polytope and Theorem \ref{Theorem1} gives rise to the following corollary. 

\begin{cor} \label{Theorem_polytope}
	Consider Assumption \ref{polytope} and any $A_M, B_M$ as in Definition \ref{def3}. Fix $\beta \in (0,1)$ and define the violation level  $\epsilon: \{0,...,M\} \rightarrow [0,1]$ as a function such that
	\begin{align}
	\epsilon(M)=1 \ \text{and} \  \sum_{k=0}^{M-1} {M \choose k}(1-\epsilon(k))^{M-k}=\beta.
	\end{align} 
	We have that 
	\begin{align*}
	\mathbb{P}^{M} \Big \{\{\theta_m\}_{m \in \mathcal{M}} \in \Theta^{M}:~\mathbb{V}(\Pi_M)> \epsilon(k^*) \Big \} \leq \beta,
	\end{align*}
	where $k^*=F_M$ is the number of facets  of $\Pi_M$.
\end{cor}
Note that, even though during the proof of our theorem we also use the vertices of the polytope, only the number of facets is needed  to provide probabilistic guarantees for the entire feasibility region. This feature is appealing from a computational point of view as, in most practical cases, the constructed polytope has a significantly smaller number of facets than extreme points.
To illustrate this, consider a finite horizon multi-agent control problem with $N$ agents, where each agent's decision is subject to upper and lower bounds at each time instance $t \in \{1,\ldots,n\}$. Hence, for a multi-sample $\{\theta_m\}_{m=1}^M \in \Theta^M$,  the feasibility domain $\prod_{t=1}^{n}\prod _{i=1}^{N}\bigcap_{m=1,\dots,M}[\underline{x}_t^i (\theta_m),\overline{x}_t^i(\theta_m)]$ of the problem is a hyperrectangle whose number of facets $F=2Nn$ grows linearly with respect to the number of decision variables, while the number of vertices is given by $V=2^{Nn}$, which grows at an exponential rate with respect to $Nn$. Such constraints arise in several applications including the electric vehicle scheduling problem of Section \ref{EVcharging}.   \par
Note that, as the dimension of the decision vector increases, evaluating the minimal support subsample becomes computationally challenging. However, several efficient algorithms have been proposed for detecting redundant constraints out of the initial set of affine constraints. The currently fastest algorithm for redundancy detection is Clarkson's algorithm \cite{Clarkson}. 
Reducing the computational complexity of Clarkson's algorithm is still an active research area in computational geometry and combinatorics. One recent noteworthy attempt can be found in \cite{May}. \par

\section{Optimization programs with uncertainty in the cost} \label{Optimization}
\subsection{Optimization setting}
In this section we show that for a specific class of problems frequently arising in practical applications, the probabilistic feasibility guarantees for the optimizers of the problem can be substantially improved by leveraging the notion of the so called support rank \cite{MorariSupportRank}.  
Assuming an uncertain cost function of a specific form and deterministic local constraints, we consider now the following program
\begin{align}
P : \min_{x \in X} J(x),
\end{align}
where $J(x)=f(x)+\max\limits_{\theta \in \Theta}g(x, \theta)$ and $f: X \rightarrow \mathbb{R}$, $g: X \times \Theta \rightarrow \mathbb{R}$ is the deterministic and the uncertain part of the cost function, respectively. In addition, the cost under study satisfies the following assumption
\begin{assumption} \label{cost_function} 
	\begin{enumerate} 
		\item  $f$ is jointly convex with respect to all agents' decision vectors, and the set $X$ is non-empty, compact and convex.
		\item  $g$ takes the aggregative form 
		\begin{align*}
		g(x, \theta)= \sum\limits_{i \in \mathcal{N}}g_i(x_i,x_{-i}, \theta) \ \text{and} \\ 
		\ g_i(x_i,x_{-i}, \theta)=x_i^T(A(\theta)\sigma(x)+b(\theta)), 
		\end{align*}
		where $\sigma: X \rightarrow \mathbb{R}^{n}$ is a mapping $(x_i)_{i  \in \mathcal{N}} \mapsto \sum\limits_{i \in \mathcal{N}} x_i$ and  $A: \Theta \rightarrow \mathbb{R}^{n \times n}$, $ b: \Theta \rightarrow \mathbb{R}^{n}$ are uncertain mappings with $A(\theta)$ being a symmetric positive semi-definite matrix for all $\theta \in \Theta$. 
	\end{enumerate}
\end{assumption}
Under Assumption \ref{cost_function} the function $J$ is convex, as the pointwise maximum of an arbitrary number of convex functions is itself a convex function \cite{Boyd2004}. From Assumption \ref{cost_function}.2  the uncertain counterpart of the cost function under study takes the form 
\begin{align*}
&g(x, \theta)=\sigma(x)^T(A(\theta)\sigma(x)+b(\theta)). 
\end{align*}  
The proposed structure captures a wide class of engineering problems, including the electric vehicle charging problem detailed in Section IV.
Since $g$ is convex, using an epigraphic reformulation we recast $P$  to the equivalent semi-infinite program
\begin{align}
P^{'}: &\min\limits_{x \in X, \gamma \in \mathbb{R}} f(x)+\gamma \nonumber \\
& \text{subject to} \ h(x,\gamma, \theta) \leq 0, \ \forall \ \theta \in \Theta \label{constraint},
\end{align}
where $h(x,\gamma, \theta)=g(x,\theta) -\gamma$.
In addition, if $(x^*, \gamma^*)$ is the optimal solution of problem $P^{'}$, then $x^*$ is the optimal solution of the original problem $P$. Due to the presence of uncertainty and the possibly infinite cardinality of $\Theta$, problem $P^{'}$ is very difficult to solve, without imposing any further assumptions on the geometry of the sample set $\Theta$ and/or the underlying probability distribution $\mathbb{P}$. 
To overcome this issue, we adopt again a scenario-based scheme \cite{IntroScenarioApproach}. The corresponding scenario program of the uncertain semi-infinite program $P^{'}$ is thus given by
\begin{align}
P^{'}_{SC}: &\min\limits_{x \in X, \gamma \in \mathbb{R}} f(x)+\gamma \nonumber \\
& \text{subject to} \ h(x,\gamma, \theta_{m}) \leq 0, \ \forall \ m \in \mathcal{M},
\end{align}
where $\{\theta_m\}_{m \in \mathcal{M}} \in \Theta^M$ is an i.i.d. multi-sample of cardinality $M$.
For the scenario program under study, we introduce the following assumption:
\begin{assumption} \label{asscenario}
	\begin{enumerate}
		\item  For any multi-sample $\{\theta_m \}_{m \in \mathcal{M}}$, the scenario program $P^{'}_{SC}$ admits a feasible solution. 
		\item The optimal solution $(x^*, \gamma^*)$ of the  scenario program $P^{'}_{SC}$ is unique.
	\end{enumerate}
\end{assumption}
In case multiple optimal points exist, one can use a convex tie-break rule to select a unique solution.
The following concept, at the core of the scenario approach, is important for the derivation of the results in the next subsection, where agent-independent robustness certificates are provided for the optimal solution. Note that this is similar to Definition 3, but it refers now to the optimal solution and not to the feasibility region.
\begin{definition}(Support constraint \cite{CampiGaratti2008}) \label{sample}
	Fix any i.i.d. multisample $ \{\theta_{m}\}_{m \in \mathcal{M}} \in \Theta^M$  and let $x^*_0=x^*_0( \{\theta_{m}\}_{m \in \mathcal{M}})$ be the unique optimal solution of the corresponding scenario program of $P$, when all the $M$ samples are taken into account. Let $x^*_{-s}=x^*_{-s} (\{\theta_{m}\}_{m \in \mathcal{M}} \setminus \theta_{s})$ be the optimal solution obtained after removal of sample  $\theta_{s}$. If $x^*_0 \neq x^*_{-s}$ we say that the constraint that corresponds to sample $\theta_s$ is a support constraint.
\end{definition}
\subsection{Agent independent probabilistic feasibility guarantees for a unique solution} \label{agent_independent}
In many practical applications there are cases where a random constraint may leave a linear subspace unconstrained for any possible sample $\theta \in \Theta$. This observation motivated the concept of the support rank as introduced in \cite{MorariSupportRank}, which allows us to provide tighter probabilistic guarantees for the problem under study.
Let $y \in \mathbb{Y} \subseteq \mathbb{R}^{d}$ and consider the following semi-infinite optimization program 
\begin{align}
& \min\limits_{y \in \mathbb{Y}} c^Ty \nonumber \\
& \text{subject to} \  l(y,\theta)\leq 0, \ \forall \ \theta \in \Theta. \label{RandomConstraint}
\end{align}
Notice that the objective function is linear without loss of generality and in the opposite case an epigraphic reformulation could be introduced.
Denoting the collection of all the linear subspaces of $\mathbb{R}^d$ as $\mathcal{L}$, we consider all the linear subspaces $L \in\mathcal{L}$ that, under the presence of the random constraint (\ref{RandomConstraint}), remain unconstrained for any uncertainty realization $\theta \in \Theta$ and any point $y \in \mathbb{Y}$  , i.e., the set
\begin{align*}
&\mathcal{U}=\bigcap\limits_{\theta \in \Theta}\bigcap\limits_{y \in \mathbb{Y}} \{L \in \mathcal{L}: L \subset F(y,\theta)\}, \\
& \text{where} \  F(y, \theta)= \{ \xi \in \mathbb{R}^{d}: l(y+\xi,\theta)=l(y, \theta)\} .
\end{align*}
\begin{definition} (Support rank \cite{MorariSupportRank})  \label{SupportRank}\\
	The support rank $\rho \in \{0, \dots, d\}$ of a random constraint equals to the dimension of the problem $d$ minus the dimension of the maximal unconstrained linear subspace $L_{max}$, i.e, $\rho=d-\text{dim}(L_{max})$. By maximal unconstrained subspace we mean the unique maximal element $L_{max} \in \mathcal{U}$ for which $L \subseteq L_{max}$, for all $L \in \mathcal{U}$.
\end{definition} 
\par From the support rank definition (see Lemma 3.8 in \cite{MorariSupportRank}) we have that Helly's dimension can be upper bounded by the support rank instead of the dimension $d$ of the problem, which is a more conservative upper bound \cite{CampiGaratti2008}, $\zeta \leq \rho \in \{0, \dots, d\}$, where $\zeta$ denotes the support dimension, a notion similar to that of Helly's dimension extended for the case of multiple random constraints (see \cite{MorariSupportRank}). Keeping this relation in mind, our main goal is to obtain a bound for the support rank of the random constraint of problem $P^{'}$, thus improving the robustness certificates of its optimal solution.
The following proposition aims at finding an upper bound for the support rank $\rho$ of the random constraint  (\ref{constraint}), that is independent from the number of agents involved in the optimization program.

\begin{proposition} \label{theorem}
	Under Assumptions \ref{cost_function} and \ref{asscenario},  the support rank $\rho$ of the random constraint (\ref{constraint}) in $P^{'}$,  has an agent independent upper bound, and in particular, $\rho \leq n+1$.
\end{proposition}
\begin{proof}
	The dimension of the problem under study is $d=nN+1$, due to the presence of the epigraphic variable. Let $\mathcal{L}$ be the collection of all linear subspaces in $ \mathbb{R}^{nN+1}$. We aim at finding the dimension of the subspace that remains unconstrained  for a scenario program subject to the random constraint $h(x,\gamma, \theta) \leq 0$ for any uncertain realisation $\theta \in \Theta$ and any decision vector $(x,\gamma) \in X \times \mathbb{R}$. We first define the collection of linear subspaces that are contained in all the sets $F(x,\gamma,\theta)$: 
	\begin{align*}
	\mathcal{U}=\bigcap_{\theta \in \Theta}\bigcap_{(x,\gamma) \in \mathbb{R}^{nN+1}} \{L \in \mathcal{L}: L \subset F(x,\gamma,\theta)\}, \text{where}
	\end{align*}
	\begin{align*}
	\ F(x,\gamma, \theta)=& \{ (\xi,\xi') \in \mathbb{R}^{nN+1}: \\ &h(x+\xi,\gamma+\xi',\theta)=h(x,\gamma, \theta)\} 
	\end{align*}
	In our case, $h(x+\xi,\gamma+\xi',\theta)=h(x,\gamma, \theta)$ yields:
	\begin{align*}
	\sigma(x+\xi)^T&(A(\theta)\sigma(x+\xi)+b(\theta))-(\gamma+\xi')=\\
	=\sigma(x)^T(&A(\theta)\sigma(x)+b(\theta))-\gamma,  \nonumber \\  
	\iff \sigma^T(x)A&(\theta)\sigma(\xi) \nonumber+\sigma^T(\xi)A(\theta)\sigma(x)+\\
	\sigma^T(\xi)A&(\theta)\sigma(\xi)+\sigma^T(\xi)b(\theta)-\xi'=0 \nonumber,  \\ 
	\iff \sigma^T(\xi)A&^T(\theta)\sigma(x) \nonumber+\sigma^T(\xi)(A(\theta)\sigma(x)+b(\theta))+\\
	\sigma^T(\xi)A&(\theta)\sigma(\xi)-\xi'=0,
	\end{align*} 
	where the first equivalence stems from the fact that $\sigma(x+\xi)$ is linear with respect to its arguments, and the second one after some algebraic rearrangement.
	Note that each of the terms above is scalar, which means that it is equal to its transpose for all  $x \in X$ and $\theta \in \Theta$, while by Assumption \ref{cost_function},  $A^T(\theta)=A(\theta)$ for any $\theta \in \Theta$. As such,
	\begin{align}
	&\sigma^T(\xi)(2A(\theta)\sigma(x)+b(\theta))+\sigma^T(\xi)A(\theta)\sigma(\xi)-\xi'=0. \label{equation2} 
	\end{align}
	Using the equalities 
	$\sigma^T(\xi)(2A(\theta)\sigma(x)+b(\theta))=(\mathds{1}_{1 \times N} \otimes (2A(\theta)\sigma(x)+b(\theta))^T)\xi$ and $\sigma^T(\xi)A(\theta)\sigma(\xi)=\xi^T (\mathds{1}_{N \times N} \otimes A(\theta))\xi$, where $\mathds{1}_{1\times N}$ denotes a row vector with all elements being equal to one and $\otimes$ denotes the Kronecker product, (\ref{equation2}) can be written in the following form:
	\begin{align}
	& (\mathds{1}_{1 \times N} \otimes (2 \sigma^T(x)A(\theta)+b^T(\theta)))\xi + \nonumber \\
	&\xi^T (\mathds{1}_{N \times N} \otimes A(\theta))\xi-\xi'=0, \label{new}
	\end{align}
	Let $\tilde{C}: X \times \Theta \rightarrow \mathbb{R}^{nN}$, $\tilde{A}: \Theta \rightarrow \mathbb{R}^{nN \times nN}$, where  $\tilde{C}(x,\theta) = \mathds{1}_{1 \times N} \otimes (2 \sigma^T(x)A(\theta)+b^T(\theta))$ and $\tilde{A}(\theta)=\mathds{1}_{N \times N} \otimes A(\theta)$, respectively. Then, equation (\ref{new}) can be written as:
	\begin{align*}
	&\tilde{C}(x,\theta)\xi+\xi^T\tilde{A}(\theta)\xi-\xi'=0, \nonumber \\
	\iff& \begin{pmatrix} \tilde{C}(x,\theta) & -1\end{pmatrix} \begin{pmatrix}\xi \\ \xi'\end{pmatrix} \\ +
	&\begin{pmatrix}\xi & \xi'\end{pmatrix}\begin{pmatrix}\tilde{A}(\theta) & 0_{nN \times 1} \\  0_{1 \times nN} & 0\end{pmatrix}\begin{pmatrix}\xi \\ \xi'\end{pmatrix}=0, \nonumber \\
	\iff& V(x,\theta)w+w^TP(\theta)w=0, 
	\end{align*}
	where, $V(x,\theta)=\begin{pmatrix}\tilde{C}(x,\theta) & -1\end{pmatrix}$, $P(\theta)=\begin{pmatrix}\tilde{A}(\theta) & 0_{nN \times 1} \\  0_{1 \times nN} & 0\end{pmatrix}$ and $w=\begin{pmatrix}\xi \\ \xi'\end{pmatrix}$. 
	We need to find an unconstrained linear subspace that is a subset of $F(x,\gamma, \theta)$. We define \\
	\begin{align*}
	L(x,\gamma, \theta)&=\{ w \in \mathbb{R}^{nN+1}: \begin{pmatrix} P(\theta)  \\  V(x,\theta) \end{pmatrix}w=0 \}.
	\end{align*}
	We can easily see that $L(x,\gamma,\theta) \subset F(x,\gamma, \theta)$. As such, the random constraint $h(x,\gamma,\theta) \leq 0$ cannot constrain any of the dimensions of $L(x,\gamma,\theta)$ (also denoted as $L$ for simplicity).
	Let $Q(x,\gamma,\theta)=\begin{pmatrix} P(\theta)  \\  V(x,\theta) \end{pmatrix}$. Then $L(x,\gamma,\theta)=\text{nullspace}(Q(x,\gamma, \theta))$ and from nullity-rank theorem \cite{axler} we have that $\text{dim}(L(x,\gamma,\theta))=nN+1-\text{rank}(Q(x, \gamma, \theta))$.
	Since $\text{rank}(P(\theta))=n$ and $\text{rank}(V(x,\theta))=1$, this means that $\text{rank}(Q(x, \gamma, \theta))=n+1$, which implies that $\text{dim}(L(x,\gamma,\theta))=nN+1-(n+1)$. Notice that the unconstrained subspace that we chose may not be the maximal one. This means that the support dimension is $\rho=nN+1-\text{dim}(L_{max}) \leq nN+1- \text{dim}(L)=n+1$, thus concluding the proof.
\end{proof}

An immediate consequence of  Proposition \ref{theorem} when combined with Theorem 4.1 of \cite{MorariSupportRank} is the following theorem.
\begin{thm} \label{corollary3}
	Let $(x^*,\gamma^*)$ denote the optimal solution of the scenario program $P^{'}_{SC}$. Under Assumptions \ref{cost_function} and \ref{asscenario} we have that
	\begin{align}
	\mathbb{P}^{M}\{\{\theta_m\}_{m \in  \mathcal{M}}&\in \Theta^M: \nonumber \\
	&\mathbb{P}( \theta \in \Theta: h(x^*,\gamma^*,  \theta)>0 ) > \epsilon\} \leq \beta, \label{probability}  \\
	&\text{where} \ \beta= \sum\limits_{j=0}^{n}{M \choose j}\epsilon^j(1-\epsilon)^{M-j} \label{beta}.
	\end{align}
\end{thm}
The bound obtained from Theorem \ref{corollary3} constitutes a major improvement for this class of problems, since, irrespective of the number of agents $N$, the same number of samples $M$ is required to provide identical robustness certificates given a local decision vector of size $n$. The proof of this theorem, is a direct application of the scenario approach theory (see Theorem 2.4 in \cite{CampiGaratti2008}, and Theorem 4.1 in \cite{MorariSupportRank}, where the number of support constraints is replaced by the obtained bound for the support rank, namely, $n+1$). Note that in the absence of Proposition 1, a direct application of the scenario approach theory \cite{CampiGaratti2008} to the problem under consideration would still result in (\ref{probability}), however, (\ref{beta}) would be replaced by
\begin{align}
\beta= \sum\limits_{j=0}^{nN}{M \choose j}\epsilon^j(1-\epsilon)^{M-j} \label{beta2},
\end{align} 
where the dependence of the guarantees on the number of agents $N$ is apparent.

Corollary \ref{corollary1} establishes a link between Theorem \ref{corollary3} and the initial program under study $P$ by providing probabilistic performance guarantees for the optimal solution. Specifically, it quantifies, in an \emph{a priori} fashion, the probability that the cost that corresponds to the optimal value $x^*$ of $P_{SC}'$ will deteriorate, when a new sample $\theta \in \Theta$ is encountered.  To formalise this,  with a slight abuse of notation let $J(x)=J(x,\{\theta_m\}_{m \in  \mathcal{M}})$ be the cost function of the corresponding scenario program of program $P$ and $J^{+}(x)=J(x,\{\theta_m\}_{m \in  \mathcal{M}}\cup\{\theta\})$ the cost defined over $M+1$ scenarios by taking into account the new sample $\theta$. 
\begin{cor} \label{corollary1} 
	Under Assumptions \ref{cost_function} and \ref{asscenario} we have that 
	\begin{align}
	\ \mathbb{P}^{M}\{&\{\theta_m\}_{m \in \mathcal{M}}\in \Theta^M:  \nonumber \\
	& \ \ \ \ \ \ \ \ \mathbb{P}( \theta \in \Theta: J^+(x^*)> J(x^*))> \epsilon\} \leq \beta, \\ 
	&\text{where} \ \beta= \sum\limits_{j=0}^{n}{M \choose j}\epsilon^j(1-\epsilon)^{M-j} 
	\end{align}
\end{cor} 
\begin{proof} 
	Let $(x^*,\gamma^*)$ be the optimal solution of program $P^{'}$, which implies that $\gamma^*=\max\limits_{m \in \mathcal{M}}g(x^*, \theta_m)$. As such,   
	\begin{align}
	\mathbb{P}( \theta \in \Theta: & \ h(x^*,\gamma^*, \theta)>0)= \mathbb{P}( \theta \in \Theta: g(x^*, \theta)>\gamma^*)= \nonumber \\
	\mathbb{P}( \theta \in \Theta: & \ g(x^*, \theta)>\max\limits_{m \in \mathcal{M}}g(x^*, \theta_m))= \nonumber  \\
	\mathbb{P}( \theta \in \Theta: & \  \max \{g(x^*, \theta), \max\limits_{m \in \mathcal{M}}g(x^*, \theta_m)\}>\max\limits_{m \in \mathcal{M}}g(x^*, \theta_m))=\nonumber  \\
	\mathbb{P}( \theta \in \Theta: & \ J^{+}(x^*)> J(x^*)) \label{cost},
	\end{align}
	where the second equality follows from the fact that $\gamma^*=\max\limits_{m \in \mathcal{M}}g(x^*, \theta_m)$, and the last one from the definitions of $J$ and $J^+$. Direct substitution of  (\ref{cost}) in (\ref{probability}) of Theorem \ref{corollary3} concludes then the proof..
\end{proof}
\section{Numerical Study} \label{EVcharging}
\subsection{Probabilistic guarantees for all feasible electric vehicle charging schedules}
In the following set-up, a cooperative scheme is considered, where  agents-vehicles  minimize a common electricity cost, while their charging schedules are subject to constraints.
However, most of the work up to this point assumed that these constraints are purely deterministic \cite{Callaway1}, \cite{Dario2019}, \cite{deori2018a}. We extend this framework by imposing uncertainty on the constraints by considering the following program
\begin{align}
&\min_{x \in \mathbb{R}^{nN}} J(x)  \ \text{subject to} \nonumber \\
&x \in  \bigcap\limits_{\theta \in \Theta}\prod_{i \in \mathcal{N}} \{x_i \in [\underline{x}_i (\theta),\overline{x}_i(\theta)]: \sum_{t=1}^{n}x_i^{(t)} \geq E_i(\theta)  \} .  \label{evgames}
\end{align} 
The two main requirements for the operation of the system under study, namely, the lower and upper bounds imposed on the power rate of each vehicle and the total energy level to be achieved at the end of charging, can be modelled as constraints of affine form.
The variables $x_i=(x_i^{(t)})_{t=1}^n$ denote the charging schedule for all time instances $t \in \{1, \dots, n\}$.
\par
The corresponding scenario program of (\ref{evgames}) is given by
\begin{align}
&\min_{x \in \mathbb{R}^{nN}} J(x)  \ \text{subject to} \nonumber \\
&x \in  \bigcap\limits_{m \in \mathcal{M}}\prod_{i \in \mathcal{N}} \{x_i \in [\underline{x}_i (\theta_m),\overline{x}_i(\theta_m)]: \sum_{t=1}^{n}x_i^{(t)} \geq E_i(\theta_m) \}.  \label{ev}
\end{align} 
The cost function  $J$ is allowed to have any arbitrary form and can even be affected by uncertainty.
In our set-up  we assume that the upper and lower bounds of the charging rate, $(\overline{x}_i(\theta_u))_{i \in \mathcal{N}}$, $  (\underline{x}_i(\theta_l))_{i \in \mathcal{N}}\in \mathbb{R}^{nN}$ are affected by the uncertain parameters  $\theta_{u}, \theta_{l} \in \mathbb{R}^{nN}$, respectively, with uncertainty representing volatile grid power restrictions. Each of the parameters' elements is extracted according to the same probability distribution $\mathcal{N}(0,0.5)$, where $\mathcal{N}(\mu,\sigma)$ is a Gaussian distribution with mean $\mu$ and standard deviation $\sigma$. The distribution is truncated by a prespecified quantity to avoid infeasibility  issues. 
We further assume that the uncertainty is additive, i.e., 
$\overline{x}(\theta_u)=\overline{x}^{nom}+\theta_{u}$ and $\underline{x}(\theta_l)=\underline{x}^{nom}+\theta_{l}$, where each element of $\overline{x}^{nom}$ is drawn from a uniform probability distribution with support $[10,20]$ kW  and $\underline{x}^{nom}$ is set to $2$ kW. Finally,  the energy capacity of the battery can be affected by a variety of factors such as battery aging and lithium plating for Li-ion batteries. These phenomena can have an important effect on the amount of energy required by each vehicle to fully charge, thus imposing uncertainty on the final energy level to be achieved by each of them by the end of the charging cycle. In our set-up, the uncertainty in the total energy $E=(E_i)_{i \in \mathcal{N}}$ (in kWh) of each vehicle at the end of charging is yet again assumed to be additive, i.e., $E=E^{nom}+\theta_e$, where $\theta_e  \in \mathbb{R}^{N}$ and its elements are extracted according to the probability $\mathcal{N}(0,1)$ and $E_i^{nom} \in \mathbb{R}$ is the nominal final energy demand of each agent $i \in\mathcal{N}$. The uncertainty vector is given by $\theta=[\theta_{u}, \theta_{l}, \theta_{e}]\in \mathbb{R}^{N(2n+1)}$. \par

\begin{figure} 
	\centering
	\includegraphics[scale=0.7]{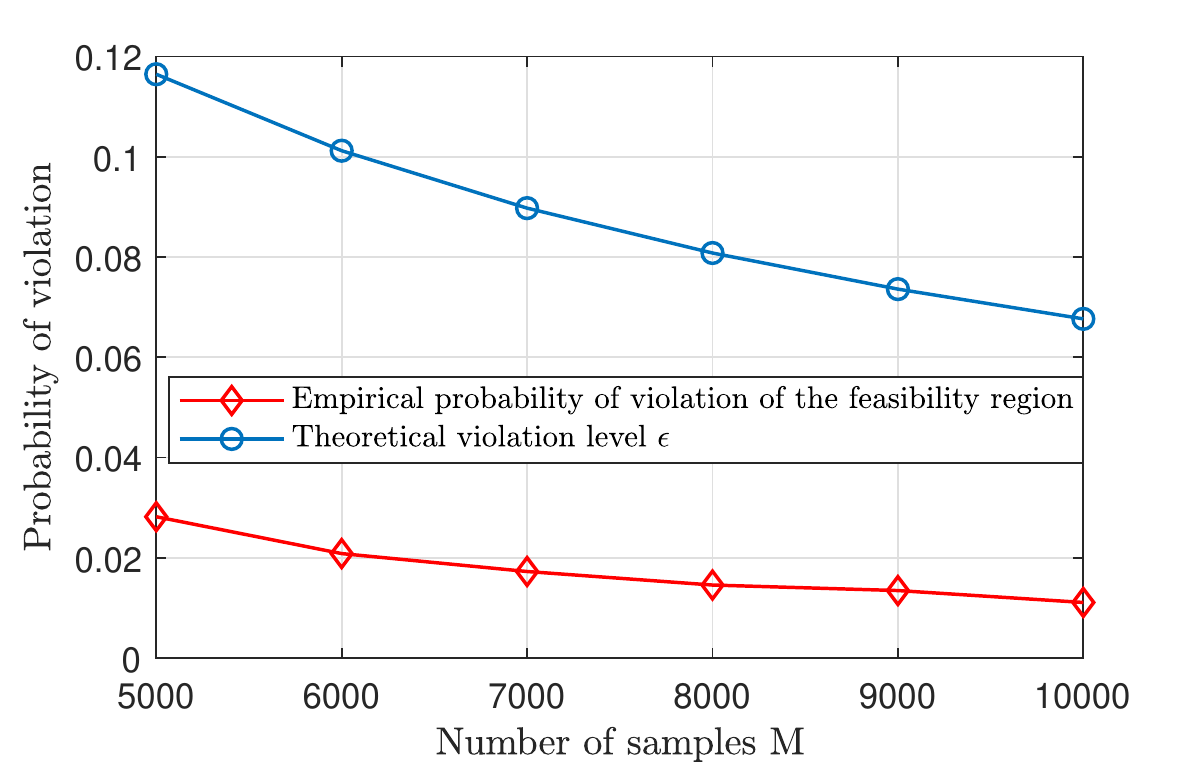}
	\caption{Empirical probability of violation of the feasibility region (red line) versus the the theoretical violation level for a different number of samples $M=\{5000,6000,\dots, 10000\}$. To calculate the probability of violation a total of $M_{test}=40000$ samples is used.} \label{feasible22}
\end{figure}
Considering $N=5$ vehicles and $n=12$ time slots we construct the feasibility region of the corresponding scenario program (\ref{ev}).
Using $M_{test}=40000$ test samples we empirically compute the  probability  that a new yet unseen constraint will be violated by at least one element of the feasible region  and compare it with the theoretical violation level $\epsilon(k)$ from Theorem \ref{Theorem1}. The results are shown in Figure 2, where the red line corresponds to the empirical probability and the blue line corresponds to the theoretical bound. Note that an upper bound for the theoretical violation level $\epsilon(k)$ can be obtained by counting the number of the facets $k=F_M$ of the feasible set or by leveraging the geometry of our numerical example to provide an upper bound for $F_N$, that is $F_N \leq 2nN+N$. This bound can be easily derived noticing that the feasible region is in fact a cartesian product of $N$ rectangles intersected by a halfspace that corresponds to the energy constraint. As such the worst case number of facets for the entire polytope in our example is $N(2n+1)$.   \par 
In general, for samples that give rise to more than one affine constraints, the number of facets constitutes only an upper bound for the cardinality of the minimal support subsample. This bound is tight only in the case when there is a one to one correspondence between a sample realization and a scalar-valued constraint. This implies that the guarantees for the feasible subset of the problem under study can be significantly tighter. 
The reason behind the use of the looser bound $k=2nN+N$ in our example lies in the fact that its quantification is straightforward and the use of a support subsample function is thus not required. In other cases, however, where an upper bound for $k$ is absent, the methodologies for the detection of redundant affine constraints provided by \cite{May} and references therein can be used as minimal support subsample functions. \par
\begin{figure}[h]
	\centering
	\includegraphics[scale=0.7]{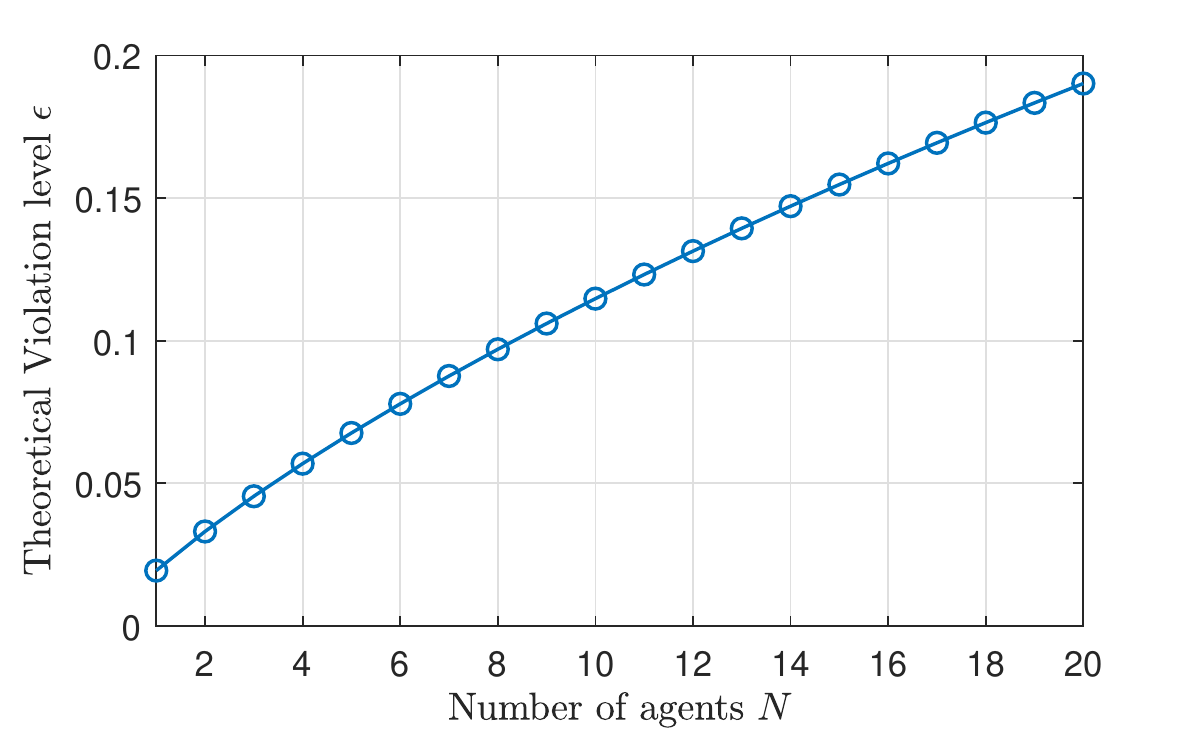}
	\caption{The violation level $\epsilon$ with respect to the number of agents $N$ for $\beta=10^{-6}$ and a multi-sample of size $M=10000$. Note how the violation level for a fixed number of samples is agent dependent, i.e., the guarantees become more conservative as the number of agents increases.} \label{feasible22}
\end{figure}
Figure \ref{feasible22} illustrates the violation level $\epsilon(k)$ with respect to the number of agents $N$ for a fixed confidence level $1-\beta$, a fixed multi-sample of size $M=10000$ and $k^*=2Nn+N$. It is clear that for the same number of samples the probabilistic guarantees provided by Theorem \ref{Theorem1} for all feasible solutions become more conservative as the number of vehicles in the fleet increases. This means that more samples are required to provide the same probabilistic guarantees as the dimension grows.

\subsection{Agent independent robustness certificates for the optimal electric vehicle charging profile}

In this set-up  the cost to be minimized is influenced by the electricity price, which in turn is considered to be a random variable affected by uncertainty. Uncertainty here refers to price volatility. All electric vehicles cooperate with each other choosing their charging schedules so as to minimize the total uncertain electricity cost, while satisfying their own deterministic constraints. To this end, we consider the following uncertain electric vehicle charging problem
\begin{align}
P_{EV}:& \min\limits_{x \in \mathbb{R}^{nN}} f(x)+ \max\limits_{\theta \in \Theta}g(x, \theta), \nonumber \\
& \text{subject to} \ x_i \in [\underline{x}_i, \overline{x}_i], \sum\limits_{t=1}^{n}x_i^{(n)} \geq E_i \ \text{for all} \  i \in \mathcal{N},
\end{align}
where  $f(x)=\sum\limits_{i \in \mathcal{N}} f_i(x_i,x_{-i})=\sigma(x)^Tp_0(\sigma(x))$ is the deterministic part of the electricity cost that depends on a nominal electricity price $p_0(\sigma(x))=A_0\sigma(x)+b_0$ that is, in turn, a function of the aggregate consumption of the vehicles. $g(x, \theta)=\sum\limits_{i \in \mathcal{N}}g_i((x_i,x_{-i}, \theta)= \sigma(x)^Tp(\sigma(x), \theta)$ constitutes the uncertain part of the electricity cost, where the price $p(\sigma(x), \theta)=A(\theta)\sigma(x)+b(\theta)$ is additionally affected by the uncertain parameter $\theta$ extracted from the support set $\Theta$ according to a probability distribution $\mathbb{P}$, where $\Theta$ and $\mathbb{P}$ are considered unknown. The elements of $A_0 \in \mathbb{R}^{n \times n}$ and $b_0 \in \mathbb{R}^n$ are deterministic with $A_0$ being a symmetric positive semi-definite matrix, while the uncertain mappings $A: \Theta \rightarrow \mathbb{R}^{n \times n} $ and  $b: \Theta \rightarrow \mathbb{R}^{n} $ are defined as in Section III. The vectors $\underline{x}_i, \overline{x}_i \in \mathbb{R}^{n}$ constitute the lower and upper bound of the charging rate of vehicle $i \in \mathcal{N}$, respectively, while $E_i \in \mathbb{R}$ is the final energy to be achieved by each vehicle $i \in \mathcal{N}$ by the end of the charging cycle. \par
Following the same lines as in Section III, we apply an epigraphic reformulation and use samples for $\Theta$ to obtain the following scenario program
\begin{align}
P_{EV}^{sc}:& \min\limits_{(x,\gamma) \in \mathbb{R}^{nN+1}} f(x)+ \gamma, \\
\text{subject to}& \ x_i \in [\underline{x}_i, \overline{x}_i], \sum\limits_{t=1}^{n}x_i^{(n)} \geq E_i, \ \text{for all} \  i \in \mathcal{N}, \nonumber \\
& g(x, \theta_m) \leq \gamma, \ \text{for all} \ m \in \mathcal{M}. \nonumber
\end{align} 
In our set-up $A(\theta) \in \mathbb{R}^{n \times n}$ is assumed to be a diagonal matrix with non-negative diagonal elements for any uncertain realization $\theta \in \Theta$. The diagonal elements  of $A(\theta)$ and the elements of $b(\theta) \in \mathbb{R}^{n}$ are extracted according to uniform distributions. For each agent $i \in \mathcal{N}$ the upper bound $\overline{x}_i$ takes a random value in the set $[6,15]$ kW, the lower bound $\underline{x}_i$ is set to 2 kW and the final energy to be achieved by the end of  the charging cycle is appropriately chosen to be feasible, considering the number of timesteps $n$ and the upper bound
of the power rate of each agent. $A_0 \in \mathbb{R}^{n \times n}$ is assumed to be a diagonal matrix, whose diagonal entries are all set to $0.01$ and $b_0$ is derived by rescaling a winter weekday demand profile in the UK. \par
\begin{figure}
	\centering
	\includegraphics[scale=0.7]{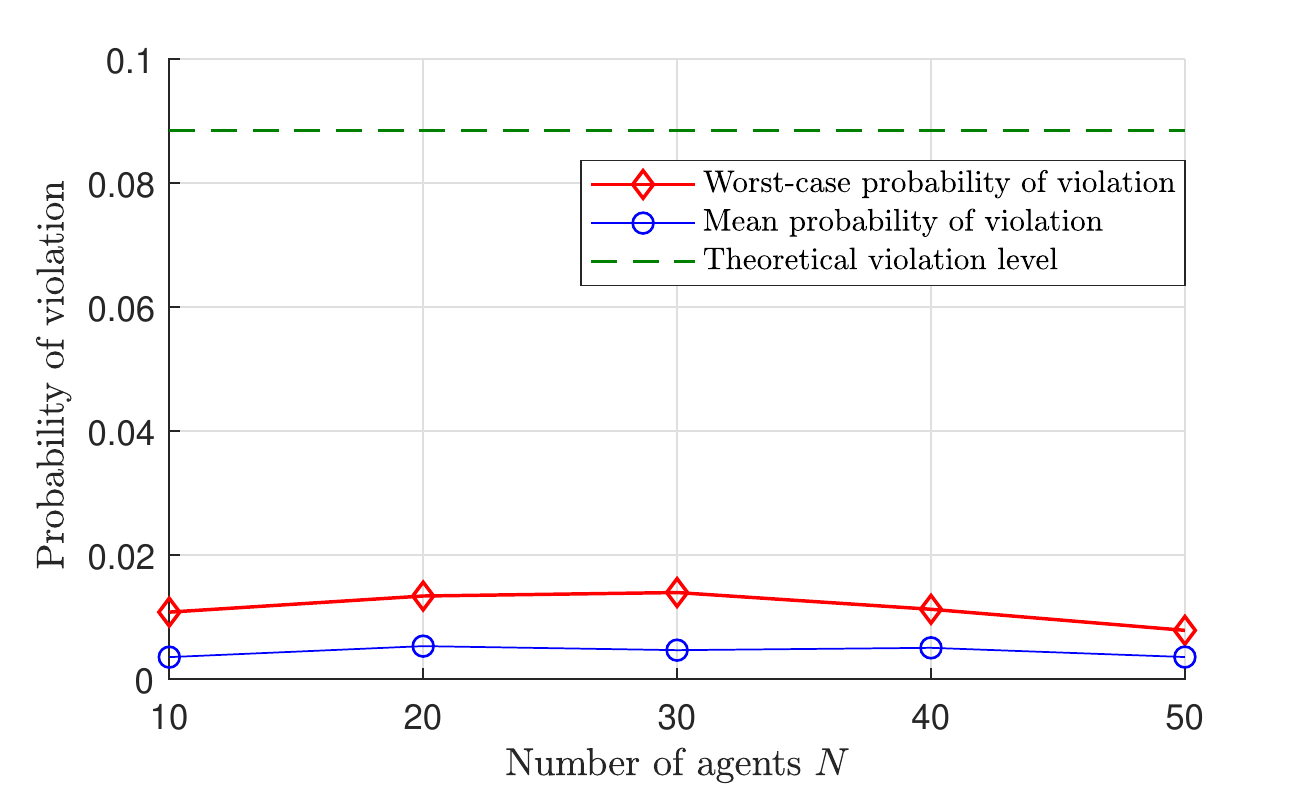}
	\caption{Mean and worst-case empirical probability of violation of the optimal solution with respect to the number of agents versus the theoretical violation level $\epsilon = 0.0885$. The number of samples used is M = 500 and $\beta = 10^{-6}$. By drawing a different multi-sample for each choice of the number of agents $N = \{10, 20, 30, 40, 50 \}$, we solve the corresponding
		scenario program for a fixed number of time slots $n = 12$. We then repeat this process 20 times (note that the multi-sample used for each repetition is also different) and compute the empirical probability of violation of the obtained optimal solutions, using $M_{test} = 100000$ test samples.}
\end{figure}
\begin{figure}
	\centering
	\includegraphics[scale=0.7]{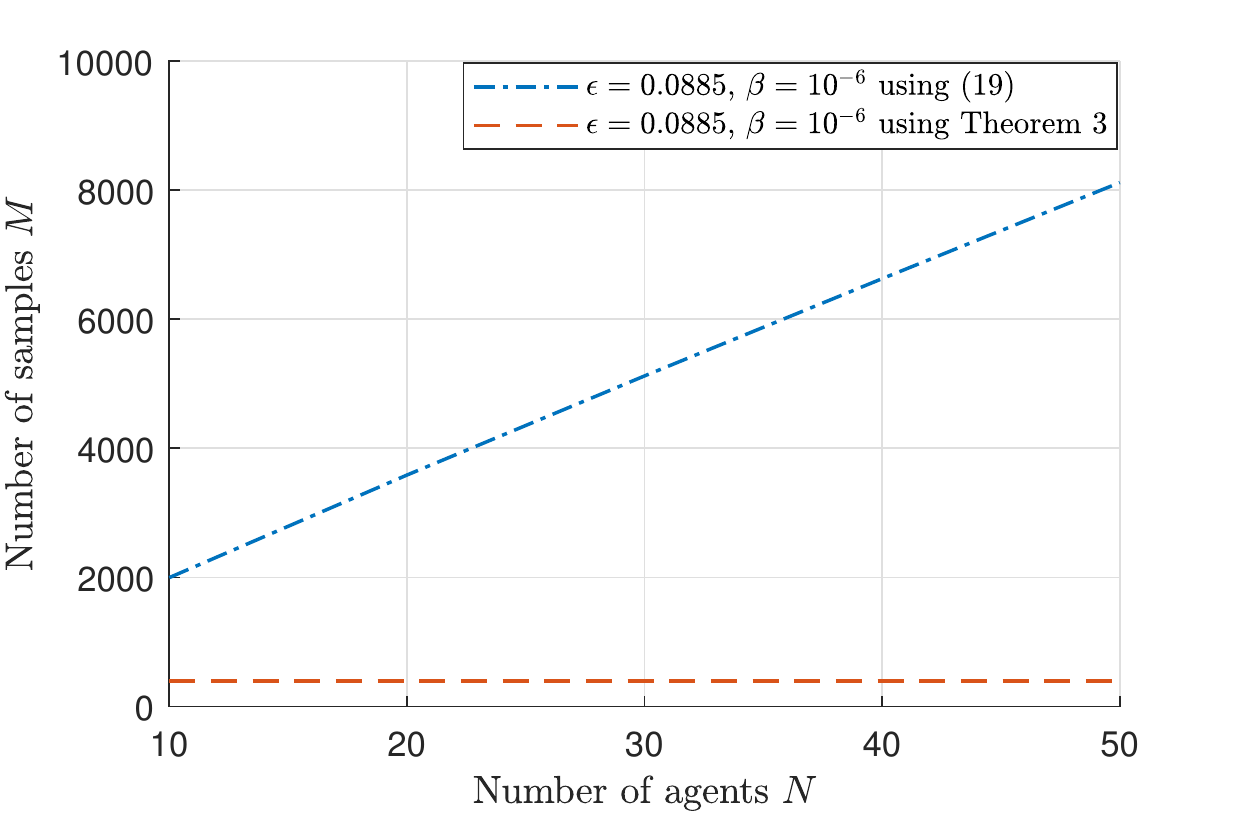}
	\caption{The number of samples required with respect to the number of agents $N = \{10,\dots,50\}$ using the results of Theorem 3 versus the one that would have been obtained if (19) is used instead. We consider a charging cycle of duration $n = 12$. The red line corresponds to Theorem 3, while the blue line corresponds to (19).}
\end{figure}
Note that our results can be used alongside any optimization algorithm irrespective of its nature, i.e.,  centralised, decentralised or distributed; here we solved the problem in a centralised fashion.
The number of samples we use for each problem is $M=500$. By fixing $\beta=10^{-6}$ and using the bound 
\begin{align}
\epsilon= \frac{2}{M}(\text{ln}\frac{1}{\beta}+n\text{ln}2),
\end{align}
which is a sufficient condition (see \cite[(p.42)]{IntroScenarioApproach}) for satisfaction of (\ref{beta}), we obtain the theoretical violation level $\epsilon=0.0885$. Note that the dimension we use to provide probabilistic guarantees for the optimal solution is set, in accordance to Theorem 3 to $n+1$ instead of $nN+1$, which circumvents the computational issues related to the rapid surge in dimension due to the multiplication of the number of agents with  the number of time slots. \par
By drawing a different multi-sample for each choice of
the number of agents $N \in \{10,20,30,40,50\}$ we solve the corresponding
scenario program for a fixed number of time slots
$n = 12$. We then repeat this process 20 times (note that the
multi-sample used for each repetition is also different) and
compute the empirical probability of violation of the obtained
optimal solutions, using $M_{test} = 100000$ test samples each
time. The mean and worst-case empirical probability of violation is depicted
in Figure 4 in comparison with the theoretical violation level
$\epsilon$.  The empirical values are always below the theoretical
level of violation, which is constant with the number of
agents due to the agent independent nature of our Theorem 3. In addition, the trend in Figure 4 shows, as expected
by Theorem 3, that the number of agents does not affect the
empirical probability of violation. \par
This result highlights the fact that, for fixed number of time periods $n$,  the number of samples $M$ required to provide identical probabilistic guarantees, as the size of the fleet of electric vehicles increases, remains constant.  This is illustrated in Figure 5, where we show the number of
samples required (for $\epsilon = 0.0885$, $\beta = 10^{-6}$ and $n=12$) using the
results of Theorem 3 versus the number of samples needed
to provide the same robustness certificates using the classic
results in scenario approach for a different number of agents
$N = 10,\dots,50$. The red line corresponds to Theorem 3 and shows the agent independent nature of our guarantees, while the blue line corresponds to the conservative agent dependent result of (19).

\section{Concluding remarks}
We first considered a general class of optimization programs with an arbitrary cost function and uncertain convex constraints and provided \emph{a posteriori} bounds for the probability of violation of all feasible solutions. We then focused on a different class of multi-agent programs that involved an uncertain aggregative term and deterministic constraints. For such problems we provided agent independent probabilistic guarantees for the optimal solution in an \emph{a priori} fashion. \par
Effort is being made towards extending our results to provide agent independent probabilistic guarantees in a non-cooperative set-up, giving rise to aggregative games. In addition, we aim at generalising the methodology of our first contribution to provide in an \emph{a posteriori}  fashion tighter robustness certificates that depend only on a subset of the feasible region, which circumscribes the game solutions. 


\bibliography{biblio_new}

\end{document}